\documentclass[hidelinks]{siamart220329}

\usepackage{lipsum}
\usepackage{amsfonts}
\usepackage{graphicx}
\usepackage{epstopdf}
\usepackage{enumerate}
\usepackage{hyperref}

\ifpdf
  \DeclareGraphicsExtensions{.eps,.pdf,.png,.jpg}
\else
  \DeclareGraphicsExtensions{.eps}
\fi

\newsiamremark{remark}{Remark}
\newsiamremark{hypothesis}{Hypothesis}
\crefname{hypothesis}{Hypothesis}{Hypotheses}
\newsiamthm{claim}{Claim}
\newsiamthm{example}{Example}
\newsiamthm{assumption}{Assumption}
\headers{Exact convergence rate of subgradient method}{M. Zamani, F. Glineur}

\title{Exact convergence rate of the last iterate\\in subgradient methods
}

\author{Moslem Zamani\thanks{ICTEAM/INMA, Université catholique de Louvain, B-1348 Louvain-la-Neuve, Belgium
  (\email{moslem.zamani@uclouvain.be}).}
\and Fran\c{c}ois Glineur \thanks{ICTEAM/INMA \& CORE, Université catholique de Louvain, B-1348 Louvain-la-Neuve, Belgium
  (\email{Francois.Glineur@uclouvain.be}).
  }}
\usepackage{amsopn}

\usepackage{algorithm,algpseudocode}
\algnewcommand{\Inputs}[1]{%
  \State \textbf{Inputs:}
  \Statex \hspace*{\algorithmicindent}\parbox[t]{.8\linewidth}{\raggedright #1}
}
\algnewcommand{\Initialize}[1]{%
  \State \textbf{Initialization:}
  \Statex \hspace*{\algorithmicindent}\parbox[t]{.8\linewidth}{\raggedright #1}
}

\ifpdf
\hypersetup{
  pdftitle={An Example Article},
  pdfauthor={D. Doe, P. T. Frank, and J. E. Smith}
}
\fi

\begin{document}

\maketitle

\begin{abstract}
We study the convergence of the last iterate in subgradient methods applied to the minimization of a nonsmooth convex function with bounded subgradients. 

We first introduce a proof technique that generalizes the standard analysis of subgradient methods. It is based on tracking the distance between the current iterate and a different reference point at each iteration. Using this technique, we obtain the exact worst-case convergence rate for the objective accuracy of the last iterate of the projected subgradient method with either constant step sizes or constant step lengths. Tightness is shown with a worst-case instance matching the established convergence rate. 

We also derive the value of the optimal constant step size when performing $N$ iterations, for which we find that the last iterate accuracy is smaller than $B R \sqrt{1+\log(N)/4}/{\sqrt{N+1}}$ 
, where $B$ is a bound on the subgradient norm and $R$ is a bound on the distance between the initial iterate and a minimizer.

Finally, we introduce a new optimal subgradient method that achieves the best possible last-iterate accuracy after a given number $N$ of iterations. Its convergence rate ${B R}/{\sqrt{N+1}}$ matches exactly the lower bound on the performance of any black-box method on the considered problem class. We also show that there is no universal sequence of step sizes that simultaneously achieves this optimal rate at each iteration, meaning that the dependence of the step size sequence in $N$ is unavoidable. 
\end{abstract}




\section{Introduction}\label{intro}
\subsection{Subgradient methods}

Subgradient methods are iterative techniques for solving nonsmooth convex optimization problems, studied by Shor and others in the 1960s. They are both simple and widely used, and continue to be actively studied. New variants have been recently developed that are more efficient and can handle a wider range of optimization problems, see \cite{davis2019stochastic, grimmer2023some} and the references therein.

Let $X\subseteq \mathbb{R}^n$ be a convex set and $f$ be a convex function whose domain contains $X$. Consider the following convex optimization problem 
\begin{align}\label{P}
\min_{x\in X} f(x).
\end{align}
The set of subgradients of function $f$ at a point $x$ is denoted as $\partial f(x)$, and is given by
$$
\partial f(x)=\{g \in \mathbb{R}^n \text{ such that }  f(y)\geq f(x)+\langle g, y-x\rangle \text{ holds for all }  y\in \mathrm{dom} f\}.
$$

The class of projected subgradient methods is given in Algorithm \ref{SM}, where $P_X(\cdot)$ stands for the Euclidean projection on $X$. An instance of the method requires to define the sequence of $N$ step sizes $\{ h_k \}_{1 \le k \le N}$, where $N$ is the number of iterations to perform. 

\begin{algorithm}
\caption{Projected subgradient method with generic step sizes}
\begin{algorithmic}
\smallskip
\State \textbf{Parameters:} number of iterations $N$, sequence of positive step sizes $\{ h_k \}_{1 \le k  \le N}$.
\smallskip
\State \textbf{Inputs:} convex set $X$, convex function $f$ defined on $X$, initial iterate $x^1 \in X$.
\smallskip
\State For $k=1, 2, \ldots, N$ perform the following steps:\\
\begin{enumerate}
\item
Select a subgradient $g^k\in \partial f(x^k)$.
\item
Compute $x^{k+1}=P_X\left(x^k-h_kg^{k}\right)$.
\end{enumerate}
\smallskip
\State \textbf{Output:} last iterate $x_{N+1}$ 
\end{algorithmic}
\label{SM}
\end{algorithm}

For the method to be well-defined, we will assume the following throughout the paper:

\begin{assumption}
\mbox{}
\begin{enumerate}
\item The set of subdifferential $\partial f(x)$ is nonempty for every $x \in X$.
\item The set $X$ is closed, convex and nonempty. 
\end{enumerate}
\end{assumption}
The first assumption is necessary to compute a direction for the next iterate. It holds for example if set $X$ is contained in the relative interior of the domain of function $f$. The second assumption ensures that the projection on $X$ is well-defined and unique.


\subsection{Convergence rates}
Unlike the gradient method, the subgradient method is not a descent method, meaning that inequality $f(x_{k+1}) \le f(x_k)$ does not necessarily hold at each iteration. For this reason, most convergence rates for the subgradient method describe the best iterate, or an average of the iterates, see for example \cite{boyd, lan, Nesterov}.
Convergence results typically require two more assumptions:
\begin{assumption}
Function $f$ has $B$-bounded subgradients on set $X$, meaning
\[  x\in X \text{ and }  g\in \partial f(x) \quad \Rightarrow \quad \| g\|\leq B.\]
\end{assumption}
Note that a convex function $f$ is Lipschitz continuous with modulus $B$ if its subgradients are $B$-bounded on its domain.
\begin{assumption}
Function $f$ admits a minimizer $x^*$, and the distance between initial iterate $x^1$ and $x^*$ is bounded by a constant $R$, that is
\[ \| x^1 - x^* \| \le R.\]
\end{assumption}
For example, under these assumptions, the best iterate of the subgradient method with generic positive step sizes $\{ h_k \}$ will satisfy (see e.g. \cite{boyd, lan})
\begin{equation} \label{ratebest} \min_{1 \le k \le N+1} f(x^k) - f(x^*) \le \frac{R^2 + B^2 \sum_{k=1}^{N+1} h_k^2}{2 \sum_{k=1}^{N+1} h_k}. \end{equation}
The same bound can be shown to also hold for the average iterate defined as \[ x^\mathrm{avg} =  \sum_{k=1}^{N+1} w_k x^k \quad \text{ with } w_k = \frac{h_k}{\sum_{k=1}^{N+1} h_k}.\]
Note that these rates depend on step size $h_{N+1}$ that is not used in the algorithm (i.e. they are valid for any value of $h_{N+1}>0$).
If we know constants $B$ and $R$, it can be shown that the optimal choice of step sizes, i.e. which minimizes the rate, consists of the following constant sequence \[ h_k = \frac{R}{B} \frac{1}{\sqrt{N+1}} \quad \text{which implies} \quad  \min_{1 \le k \le N+1} f(x^k) - f(x^*) \le \frac{BR}{\sqrt{N+1}}.  \]
A final remark is that this last result cannot be improved. Indeed, it is known that no subgradient method can guarantee a rate better than $\frac{BR}{\sqrt{N+1}}$ \cite{drori2016optimal}. This lower bound is actually valid for any black-box method that moves in a direction combining past subgradients at each step, see beginning of Section~\ref{Sec_opt} for more details.

\subsection{Rates on the final iterate}
From the above it appears that subgradient methods are both simple and efficient, matching the best possible convergence rate. Nevertheless, we observe two drawbacks: first, the optimal sequence of step sizes $h_k = \frac{R}{B}\frac{1}{\sqrt{N+1}}$ requires knowledge of constants $B$ and $R$ and, more importantly, depends on the number of iterations $N$. 

Second, these worst-case convergence rates only hold for the best or the average iterate, and nothing is guaranteed about the sequence of iterates, or in particular about the last iterate. As the final iterate is commonly selected in practice as the output of the subgradient method \cite{jain2019making}, it may be of interest to analyze the method with respect to the last generated iterate. It is also of interest in situations where the iterates cannot be stored, where the function value cannot be evaluated, or where the sequence of iterates computed by the subgradient method is itself under study.

The question of last-iterate convergence rates was previously raised in \cite{shamir2012open}. In \cite{nesterov2015quasi} the authors introduce a modified subgradient method with double averaging for which the whole sequence of iterates converges with the rate $O(\frac{1}{\sqrt{N}})$. Moving back to standard subgradient methods as described by Algorithm \ref{SM}, a lower bound of order $O(\frac{\log(N)}{\sqrt{N}})$ for the convergence rate for the last iterate is established in \cite{harvey2019tight} for a specific choice of step sizes $h_k = \frac{R}{B} \frac{1}{\sqrt{k}}$. The authors also prove a high probability upper bound with the same order $O(\frac{\log(N)}{\sqrt{N}})$ in the stochastic case. Finally, we can find in \cite{jain2019making} a subgradient method with a different choice of step sizes for which a $O(\frac{1}{\sqrt{N}})$ convergence rate for the last iterate is obtained when the feasible set $X$ is bounded. 


In this paper, we continue to explore this question and contribute in two ways. First, we establish in Section \ref{Sec_f} exact convergence rates for the last iterate the subgradient method with either constant step sizes or constant step lengths. These results are based on a key lemma presented in Section \ref{Sec_1}, which generalizes the standard analysis of subgradient methods. Second, we present in  Section \ref{Sec_opt} an optimal subgradient method for which the last-iterate convergence rate matches exactly the established lower bound for black-box nonsmooth convex optimization problems, namely for which we have $f(x^N) - f(x^*) \le \frac{BR}{\sqrt{N+1}}$, improving the constant in the rate of \cite{jain2019making} by an order of magnitude.

\section{Key lemma for convergence proofs}\label{Sec_1}


All convergence rates established in this paper will be derived from the following key lemma. Its proof is based on tracking the distance between the current iterate and a different reference point at each iteration ($\|x_k -  z_k\|$ in the proof below). 


\begin{lemma}\label{P_SM}
Let $f$ be a convex function and let $X$ be a closed convex set. Suppose that  $\hat x\in X$, $h_{N+1}>0$ and  $0<v_0\leq v_1\leq \dots \leq v_N\leq v_{N+1}$. If  Algorithm \ref{SM} with the starting point $x^1 \in X$  generates $\{(x^k, g^k)\}$, then
\begin{align}\label{UR_SM}
   & \sum_{k=1}^{N+1}  \left(h_kv_{k}^2-(v_k-v_{k-1})\sum_{i=k}^{N+1}  h_iv_{i} \right) f(x^{k}) -
v_0 \sum_{k=1}^{N+1}  h_kv_{k}  f(\hat x) 
 \leq \\
 \nonumber & \ \ \ \ 
  \tfrac{v_{0}^2}{2}\left \| x^{1}-\hat x\right\|^2 +\tfrac{1}{2}\sum_{k=1}^{N+1}  h_k^2v_{k}^2 \left \|g^{k}\right\|^2.
\end{align}
 \end{lemma}

\noindent Note that this inequality can also be equivalently written 
\[ \sum_{k=1}^{N+1}  c_k \bigl( f(x_k) - f(\hat{x}) \bigr) \le \tfrac{v_{0}^2}{2}\left \| x^{1}-\hat x\right\|^2 +\tfrac{1}{2}\sum_{k=1}^{N+1}  h_k^2v_{k}^2 \left \|g^{k}\right\|^2 \]
with coefficients $c_k$ defined as $c_k = h_kv_{k}^2-(v_k-v_{k-1})\sum_{i=k}^{N+1}  h_iv_{i}$, since one can show that $\sum_{k=1}^{N+1} c_k = v_0 \sum_{k=1}^{N+1} h_k v_k$ using summation by parts.

\noindent \begin{proof}
 Let $z^0=\hat x$ and $z^k$ is defined recursively as follows,
 \begin{align*}
   z^{k}=\left(1-\tfrac{v_{k-1}}{v_{k}}\right)x^{k}+\tfrac{v_{k-1}}{v_{k}}z^{k-1}& , \ \ \ \ k\in\{1, \dots, N+1\}.
 \end{align*}
 It is seen that $z^k$ may be written as a convex combination of $x^1, \dots, x^{N+1}, \hat x$. Indeed,
  $$
  z^k=\tfrac{1}{v_k}\sum_{i=1}^k (v_i-v_{i-1})x^i+\tfrac{v_0}{v_k}\hat x.
  $$
  By Jensen's inequality, we get
  \begin{align}\label{ineq1}
  \sum_{k=1}^{N+1}  h_kv_{k}^2 \left( f(z^{k})-f(x^{k}) \right) \leq &
   \sum_{k=1}^{N+1}\sum_{i=1}^{k}  h_kv_{k}(v_i-v_{i-1}) f(x^{i})+
   v_0 \sum_{k=1}^{N+1}  h_kv_{k}  f(\hat x) 
   \\
 \nonumber   &\ \ -\sum_{k=1}^{N+1}  h_kv_{k}^2 f(x^{k}) 
  \end{align}
On the other hand, by the subgradient inequality for $k\in\{1, \dots, N+1\}$, we have
 \begin{align*}
     f(z^{k})-f(x^{k}) &\geq \left\langle \sqrt{h_k}g^{k}, \tfrac{1}{\sqrt{h_k}} ( z^{k}-x^{k})\right\rangle\\
     &= \tfrac{h_k}{2}\left( \left \|g^{k}+ \tfrac{1}{h_k} ( z^{k}-x^{k})\right\|^2- 
       \tfrac{1}{h^2_k}\left \| z^{k}-x^{k}\right\|^2-
       \left \|g^{k}\right\|^2
      \right)\\
      &=
      \tfrac{h_k}{2}\left( \left \|g^{k}+ \tfrac{1}{h_k} ( z^{k}-x^{k})\right\|^2- 
      \tfrac{v_{k-1}^2}{h^2_k v_{k}^2}\left \| z^{k-1}-x^{k}\right\|^2-
       \left \|g^{k}\right\|^2
      \right)
 \end{align*}
 Due to the non-expansive property of the projection operator, we have  $\|P_X(x^k-h_kg^k)-y \|\leq \|(x^k-h_kg^k)-y \| $ for any $y\in X$. Thus, we get 
  \begin{align*}
     f(z^{k})-f(x^{k}) \geq &
      \tfrac{h_k}{2}\left \|g^{k}+ \tfrac{1}{h_k} ( z^{k}-x^{k})\right\|^2-
            \tfrac{v_{k-1}^2h_{k-1}^2}{2v_{k}^2 h_k}\left \| g^{k-1}+\tfrac{1}{h_{k-1}}(z^{k-1}-x^{k-1})\right\|^2
            \\
            & \ \ - \tfrac{h_k}{2}\left \|g^{k}\right\|^2,
 \end{align*}
 for $k\in\{2, \dots, N+1\}$. Hence, 
   \begin{align*}
     2h_kv_{k}^2 \left( f(z^{k})-f(x^{k}) \right) \geq 
    & -{h_k^2v_{k}^2}\left \|g^{k}\right\|^2
      +{h_k^2v_{k}^2}\left \|g^{k}+ \tfrac{1}{h_k} ( z^{k}-x^{k})\right\|^2
      \\
     & \ \  - {v_{k-1}^2h_{k-1}^2}\left \| g^{k-1}+\tfrac{1}{h_{k-1}}(z^{k-1}-x^{k-1})\right\|^2.
 \end{align*}
 Moreover,
 \begin{align*}
 2h_1v_{1}^2 \left(  f(z^{1})-f(x^{1}) \right) \geq 
       {h_1^2 v_{1}^2}\left \|g^{1}+ \tfrac{1}{h_1} ( z^{1}-x^{1})\right\|^2- 
      {v_{0}^2}\left \| z^0-x^{1}\right\|^2-
       {h_1^2 v_{1}^2}\left \|g^{1}\right\|^2.
 \end{align*}
 By summing  up these inequalities, we obtain 
 \begin{align}\label{ineq2}
 2\sum_{k=1}^{N+1}  h_kv_{k}^2 \left( f(z^{k})-f(x^{k}) \right) \geq &
 {h_{N+1}^2v_{N+1}^2}\left \|g^{N+1}+ \tfrac{1}{h_{N+1}} ( z^{N+1}-x^{N+1})\right\|^2
 \\
\nonumber & \ \  -{v_{0}^2}\left \| z^0-x^{1}\right\|^2-
 \sum_{k=1}^{N+1}  h_k^2v_{k}^2 \left \|g^{k}\right\|^2.
 \end{align}
 Inequalities \eqref{ineq1} and   \eqref{ineq2} imply the desired inequality and the proof is complete.
\end{proof}

Compared to the standard analysis of subgradient methods, additional flexibility is provided by the sequence of weights $\{ v_k \}$. Note that by setting $v_k=1$ for all $k$ and $\hat x=x^\star$ in \eqref{UR_SM}, we get 
\begin{align}\label{C_ineq}
    \sum_{k=1}^{N+1}  h_k \left(f(x^{k}) -f^\star\right)
 \leq 
  \tfrac{1}{2}\left \| x^{1}-x^\star\right\|^2 +\tfrac{1}{2}\sum_{k=1}^{N+1}  h_k^2 \left \|g^{k}\right\|^2.
\end{align}
from which it is straightforward to derive the standard convergence rate \eqref{ratebest} with respect to the best objective value or the average of iterates \cite{boyd, lan}.

In order to establish a last-iterate convergence rate via Lemma \ref{P_SM}, one should choose appropriate values for the $N+3$ parameters, $v_0, \dots, v_{N+1}$ and $h_{N+1}$, so that coefficients $c_k$ are zero for all $k$ except $c_{N+1}$. One can actually see with some algebra that all parameters are uniquely determined if we assign values to $v_{N+1}$, $h_{N+1}$ and the coefficient $c_{N+1}$ of $f(x^{N+1})$ in \eqref{UR_SM}. 
\section{Subgradient method with constant step sizes}\label{Sec_f}

In this section, we investigate the convergence rate of Algorithm \ref{SM} when the step size is constant for each iteration. Moreover, following the standard presentation of such convergence results, we assume that this constant step size is chosen proportionally to the ratio $\frac{R}{B}$, namely we define $h_k=\tfrac{hR}{B}$ ($k=1, \dots, N$) for some $h>0$. This normalization leads to slightly simpler expressions for the rates, which become proportional to a common factor $B R$. 

\subsection{Increasing sequences $\{s_{\alpha, k}\}_{k \ge 1}$}

Before we prove our results we need to introduce a family of real sequences.  Let $\alpha \ge 1$ be a real parameter. We define the sequence $\{s_{\alpha, k}\}_{k \ge 1}$ recursively as follows
\begin{align}\label{Rseq}
s_{\alpha, 1}=\alpha, \ \ \ \ s_{\alpha, {k+1}}=s_{\alpha, k}+\tfrac{1}{s_{\alpha, k}}.
\end{align}
The next proposition lists some properties of these sequences that will be used later.

\begin{proposition}\label{P_RS}
Any sequence $\{s_{\alpha, k}\}$ defined by \eqref{Rseq} satisfies the following for all $k$:
\begin{enumerate}[(a)]
  \item
   $
  s_{\alpha, k+1} = \alpha + \sum_{i=1}^{k}\tfrac{1}{s_{\alpha, i}}
  $
  \item
  $
  s_{\alpha, k+1}^2= \alpha^2 + 2k + \sum_{i=1}^{k}\tfrac{1}{s_{\alpha, i}^2}
  $
  \item
  $\beta > \alpha $ implies $s_{\beta, k} > s_{\alpha, k}$
    \item
    $
  \lim_{\alpha\to+\infty} s_{\alpha, k}=+\infty.
  $
\end{enumerate}
\end{proposition}
\begin{proof}
\mbox{}
\begin{enumerate}[(a)] \item This follows from telescoping in the sum of defining equalities $s_{\alpha, {i+1}}=s_{\alpha, i}+\tfrac{1}{s_{\alpha, i}}$ for $i$ ranging from $1$ to $k$.
\item  Squaring the defining equality gives   $s_{\alpha, {i+1}}^2 =(s_{\alpha, i}+\tfrac{1}{s_{\alpha, i}})^2 = s_{\alpha, i}^2+ 2 + \tfrac{1}{s_{\alpha, i}^2}$. Summing for $i$ ranging from $1$ to $k$ and telescoping provides the result.
\item This follows from the fact that $s \mapsto  s+\frac{1}{s}$ is strictly increasing when $s \ge 1$.
\item This follows from $\lim_{\alpha\to\infty} s_{\alpha, 1} = +\infty$ and the fact that each sequence $\{ s_{\alpha,k}\}$ is strictly increasing.
\end{enumerate}
\end{proof}
The sequence in the particular case $\alpha=1$ will play a central role in our convergence rates. We denote $\{s_{1, k}\}$ by $\{s_k\}$ for convenience, meaning that \[
s_{1}=1, \qquad s_{{k+1}}=s_{k}+\tfrac{1}{s_{k}} \] and provide the following estimate of its asymptotic behavior.

\begin{lemma}\label{L_RS}
For any $k\geq 2$ we have \[ \sqrt{2k} \le s_{k} \le \sqrt{2k+\tfrac{1}{2}\log{(k-1)}}.\]
\end{lemma}
\begin{proof}
  To prove the left inequality, since $s_{i+1}^2 = s_i^2 + \frac{1}{s_i^2} + 2 \ge s_i^2 + 2$, we have by induction that $s_k^2 \ge s_2^2 + 2(k-2) = 2k$ (using $s_2=2$).
    To prove the right inequality, we use (b) in Proposition \ref{P_RS} to obtain 
 \begin{align*}
    s_{k}^2=1^2 + 2(k-1) + \sum_{i=1}^{k-1}\tfrac{1}{s_i^2} \leq 2k-1 + (1 + \tfrac{1}{2}\sum_{i=2}^{k-1}\tfrac{1}{i}) \leq 2k + \tfrac12  \log(k-1),
 \end{align*}
where the first inequality follows from $s_k^2 \ge 2k$, and the second from the upper bound on harmonic numbers $\sum_{i=1}^k \frac{1}{i} \le \log(k)+1$.
\end{proof}

\subsection{Convergence rate for the last iterate}
We now turn to proving a convergence rate for the last iterate of the subgradient method with constant step sizes. With the choice of constant step size explained above $h_k = \frac{R}{B} h$ for some positive parameter $h$, the subgradient algorithm becomes
\begin{algorithm}
\caption{Projected subgradient method with constant step sizes}
\begin{algorithmic}
\smallskip
\State \textbf{Parameters:} number of iterations $N$, normalized step size parameter $h>0$
\smallskip
\State \textbf{Inputs:} convex set $X$, convex function $f$ defined on $X$ with $B$-bounded subgradients, initial iterate $x^1 \in X$ satisfying $\|x^1 - x^*\| \le R$ for some minimizer $x^*$.
\smallskip
\State For $k=1, 2, \ldots, N$ perform the following steps:\\
\begin{enumerate}
\item
Select a subgradient $g^k\in \partial f(x^k)$.
\item
Compute $x^{k+1}=P_X\left(x^k- h \frac{R}{B} g^{k}\right)$.
\end{enumerate}
\smallskip
\State \textbf{Output:} last iterate $x_{N+1}$ 
\end{algorithmic}
\label{SMC}
\end{algorithm}

Most of the effort in obtaining our convergence rate will be spent in obtaining the following lemma. 

\begin{lemma}\label{P_SM1}
Let $f$ be a convex function with $B$-bounded sugradients on a convex set $X$, and let $\alpha\geq 1$. Let $\hat{x} \in X$ be a reference point. Consider $N$ iterations of Algorithm \ref{SMC} with step size parameter $h>0$, starting from an initial iterate $x^1 \in X$ satisfying $\|x^1 - \hat{x} \| \le R$. We have that the last iterate $x_{N+1}$ satisfies
\begin{align}\label{UR_SM1}
f(x^{N+1})-f(\hat x) \leq BR\left( \tfrac{1}{2}\left(s_{\alpha, N+1} \sqrt{h}-\tfrac{1}{s_{\alpha, N+1} \sqrt{h}}\right)^2+1-Nh\right).
\end{align}
\end{lemma}

\begin{proof}
 To prove inequality \eqref{UR_SM1}, we employ Lemma \ref{P_SM}. Assume that
$$
v_k=\frac{1}{s_{\alpha, N+1-k}}, \ \ k\in\{ 0, 1, \dots, N\},
$$
and $v_{N+1}=s_{\alpha, 1}$. It is seen that $0<v_0\leq v_1\leq \dots \leq v_{N+1}$. Suppose that $h_{N+1}=\tfrac{hR}{B}$. By using Proposition \ref{P_RS}, one can verify that for $k\in\{1, \dots, N\}$, 
\begin{align*}
 v_{k}^2 -( v_k-v_{k-1} ) \sum_{i=k}^{N+1}  v_{i}
 &=\frac{1}{s^2_{\alpha, {N+1-k}}} 
 -\left( \tfrac{1}{s_{\alpha, {N+1-k}}}-\tfrac{1}{s_{\alpha, {N+2-k}}} \right)
 \left( \sum_{i=1}^{N+1-k}  \tfrac{1}{s_{\alpha, {i}}}+s_{\alpha, 1} \right)\\
 & =\frac{1}{s^2_{\alpha,{N+1-k}}} -\left( \tfrac{s_{\alpha,{N+2-k}} }{s_{\alpha,{N+1-k}} }-1 \right)=0.
\end{align*}
Furthermore,
\begin{align*}
& v_{N+1}^2 -v_{N+1}(v_{N+1}-v_{N})=s_{\alpha, 1}(\tfrac{1}{s_{\alpha, 1}})=1, \\
& v_0 \sum_{k=1}^{N+1}  v_{k}=\tfrac{1}{s_{\alpha, {N+1}} } \left(\sum_{k=1}^{N} \tfrac{1}{s_{\alpha, {1}} }+s_{\alpha, 1}\right)=1. 
\end{align*}
Hence, by  Lemma \ref{P_SM}, we obtain
\begin{align*}
f(x^{N+1})-f^\star &
 \leq  \tfrac{Rh}{2B}\sum_{i=1}^{N+1} v_k^2\left\|g^i\right\|^2+
\tfrac{Bv_0}{2Rh}\left\|x^1-x^\star\right\|^2
 \\ & \leq
  \tfrac{BRh}{2}\sum_{i=1}^{N} \frac{1}{s_{\alpha, {N+1-k}}^2 }+ \frac{BRh\alpha^2}{2}+
\frac{BR}{2hs_{\alpha, {N+1}}^2}\\
&=BR\left( \tfrac{1}{2}\left(s_{\alpha, {N+1}} \sqrt{t}-\tfrac{1}{s_{ \alpha, {N+1}} \sqrt{t}}\right)^2+1-Nh\right),
\end{align*}
where the last equality follows from Proposition \ref{P_RS}. Hence, we derive the desired inequality and the proof is complete.
\end{proof}

The following theorem now uses Lemma \ref{P_SM1} with the choice $\hat x = x^*$ to obtain a last-iterate convergence rate for the subgradient method with constant step sizes. 

\begin{theorem}\label{T2}
Let $f$ be a convex function with $B$-bounded sugradients on a convex set $X$. Consider $N$ iterations of Algorithm \ref{SMC} with step size parameter $h>0$, starting from an initial iterate $x^1 \in X$ satisfying $\|x^1 - x^*\| \le R$ for some minimizer $x^*$. 
We have that the last iterate $x_{N+1}$ satisfies
\[ f(x^{N+1})-f^\star\leq \begin{cases} BR\left( 1-Nh \right) & \text{ when } h \le \tfrac{1}{s_{N+1}^2}, 
\\ BR\left( (\tfrac{1}{2} s^2_{N+1}-N)h+ \tfrac{1}{2s^2_{N+1}h}\right) & \text{ when } h > \tfrac{1}{s_{N+1}^2}.\end{cases} \]
\end{theorem}
\begin{proof}
  We prove the theorem by plugging a suitable value for $\alpha$ into inequality \eqref{UR_SM1} written for the choice $\hat x = x^*$ , since it holds for any $\alpha\geq 1$.
 In the first case, when $h \le \tfrac{1}{s_{N+1}^2}$, we may select by Proposition \ref{P_RS} a value of $\alpha\geq 1$ such that
  $$
  s_{\alpha, N+1} \sqrt{h}-\tfrac{1}{s_{\alpha, N+1} \sqrt{h}}=0
  $$
  which leads to the desired inequality. In the second case where $h> \tfrac{1}{s_{N+1}^2}$ one can choose $\alpha=1$ to obtain the inequality. \end{proof}

It is interesting to compare this last-iterate convergence rate to the one holding for the best iterate. Plugging $h_k = \frac{R}{B} h$ into the rate \eqref{ratebest}, we obtain that \[
\min_{1 \le k \le N+1} f(x^k) - f(x^*) \le BR \Bigl(\tfrac12 \tfrac{1}{(N+1)h} + \tfrac12 h \Big).\]

Using the bounds $2N+2 \le s_{N+1}^2 \le 2N+2+\tfrac12 \log(N)$ from Lemma~\ref{L_RS}, we rewrite the rate for larger steps from Theorem~\ref{T2} in the following slightly weaker but easier to interpret form:
\[f(x^{N+1})-f(x^\star) \leq BR \left( (1 + \tfrac14 \log(N))  h  + \tfrac{1}{4(N+1)h} \right). \]
Finally, when using the standard optimal constant step size $h=\frac{1}{\sqrt{N+1}}$ we obtain 
\[f(x^{N+1})-f(x^\star) \leq \frac{BR}{\sqrt{N+1}} \left( \tfrac54 + \tfrac14 \log(N) \right), \]
which show a logarithmic loss compared to the $\frac{BR}{\sqrt{N+1}}$ rate for the best iterate.

A last interesting remark is that Lemma~\ref{P_SM1} can be used with a reference $\hat{x}$ that is not a minimizer. In essence, it shows that subgradient methods converge to any sublevel set of the objective function with the same worst-case rate, provided the constant $R$ in its numerator is taken as the distance from the initial iterate to that sublevel set.


\subsection{Tightness of the convergence rate}\label{Sec_f2}

A convergence rate is said to be exact (or tight) if there exists a problem instance achieving that rate. We now show that the convergence rates we obtained for the subgradient method with constant step sizes are exact.

In the case of shorter steps ($h \le \frac{1}{s_{N+1}^2}$), it is readily seen that the convergence rate in Theorem~\ref{T2} is exact. Indeed, it is enough to consider an unconstrained optimization univariate problem ($n=1$, $X=\mathbb{R}$) with objective function $f(x)=B|x|$ and the initial point $x^1=R$. 

In the case of longer steps  ($h >\frac{1}{s_{N+1}^2}$), the following more involved example illustrates that the convergence rate in Theorem~\ref{T2} is also exact. To establish exactness, we may assume without loss of generality that $R=B=1$. We also use $e_i$ to denote the $i$th unit vector.

\begin{example}
Let $N$ be a number of iterations and  $h > \tfrac{1}{s_{N+1}^2}$. Let $\gamma_1=1$ and define
\begin{align*}
\gamma_k=\sqrt{\prod_{i=1}^{k-1} \left(1-\tfrac{1}{s^4_{N+1-i}} \right) }, \ k\in\{2, \dots, N\}.
\end{align*}
Suppose that $\xi^1, \dots, \xi^{N+1}\in \mathbb{R}^{N+1}$ are given as follows,
\begin{align*}
\xi^k=\tfrac{1}{hs_{N+1}^2}e_1+\sqrt{ 1-\tfrac{1}{h^2s_{N+1}^4} }\left(\sum_{i=2}^{k}  \tfrac{\gamma_{i-1}}{s_{N+2-i}^2} e_i-\gamma_k e_{k+1}\right), \ k\in\{1, \dots, N\},
\end{align*}
and $\xi^{N+1}=\xi^N+2\gamma_{N}\sqrt{ 1-\tfrac{1}{h^2s_{N+1}^4} }e_{N+1}$. By the definition of $\xi^k$, it is seen 
$$
\left\|\xi^k\right\|=1, \ \ k\in\{1, \dots, N+1\}.
$$
For $k<N$ and $k<j$, we have
\begin{align}\label{Ex.2}
\left\langle \xi^k, \xi^j\right\rangle &=\left\langle \xi^k, \xi^k+\gamma_k\left(\tfrac{1}{s^2_{N+1-k}}+1\right)\sqrt{ 1-\tfrac{1}{h^2s_{N+1}^4} } e_{k+1} \right\rangle\\
\nonumber &=1-\gamma_k^2\left(\tfrac{1}{s^2_{N+1-k}}+1\right)\left(1-\tfrac{1}{h^2s_{N+1}^4}\right),
\end{align}
 and $\left\langle \xi^N, \xi^{N+1}\right\rangle=1-2\gamma_N^2\left(1-\tfrac{1}{h^2s_{N+1}^4}\right) $. Let $f: \mathbb{R}^{N+1}\to\mathbb{R}$ given by 
$$
f(x)=\max_{0\leq k\leq N+1} f^k+\langle \xi^k, x-z^k\rangle
$$
where $f^0=0$,
\begin{align*}
f^k=\tfrac{1}{s_{N+1}^2}+\left( 1-\tfrac{1}{h^2 s_{N+1}^4} \right)\sum_{i=1}^{k-1} \gamma_i^2\left( 1+\tfrac{1}{s^2_{N+1-i}}  \right) , \ \ k\in\{1, \dots, N+1\},
\end{align*}
and $\xi^0=z^0=0$, 
\begin{align*}
z^k=e^1-h\sum_{i=1}^{k-1} \xi^i, \ \ k\in\{1, \dots, N+1\}.
\end{align*}
It is readily seen that $\|\xi\|\leq 1$ for any $\xi\in \partial f(x)$ and $x\in\mathbb{R}^{N+1}$. By \eqref{Ex.2}, one can show that $f(0)= 0$. Hence, $0\in\partial f(0)$ and zero is an optimal solution of the following problem,
$$
\min_{x\in\mathbb{R}^{N+1}} f(x).
$$ 
Furthermore, 
\begin{align}
\xi^k\in\partial f(x^k), \ \  k\in\{1, \dots, N\}.
\end{align}
 After some algebraic manipulations, one can show that $f(x^{N+1})=(\tfrac{1}{2} s^2_{N+1}-N)h+ \tfrac{1}{2s^2_{N+1}h}$.   Algorithm \ref{SM} with initial  point $x^1=e_1$ and step size $h$ may generate the following points  
 $$
 x^k=z^k,  g^k=\xi^k \ \ k\in\{2, \dots, N+1\}.
 $$
 Since $f(x^{N+1})=(\tfrac{1}{2} s^2_{N+1}-N)h+ \tfrac{1}{2sq^2_{N+1}h}$, one infers the tightness of the rate for large steps in Theorem \ref{T2}.
\end{example}

\subsection{Optimal constant step size}

 In what follows, we determine the optimal constant step size for Algorithm \ref{SM}, i.e. the value of $h$ that minimizes the rate established in Theorem \ref{T2}, and derive the resulting optimal last-iterate convergence rate for this class of subgradient methods.
 
\begin{theorem}\label{T_O_S}
Let $f$ be a convex function with $B$-bounded sugradients on a convex set $X$. Consider $N$ iterations of Algorithm \ref{SMC} starting from an initial iterate $x^1 \in X$ satisfying $\|x^1 - x^*\| \le R$ for some minimizer $x^*$. The optimal value of the step size parameter $h$ is given by 
  $$
   h^\star=\frac{1}{s_{N+1}\sqrt{s_{N+1}^2-2N}},
  $$
and with that choice the last iterate $x_{N+1}$ satisfies
\[ f(x^{N+1})-f(x^\star) \leq BR \sqrt{1 - \tfrac{2N}{s_{N+1}^2}}. \]
\end{theorem}

\begin{proof}
  To get the optimal step size suffices to minimize the function
  $H:\mathbb{R}_+\to \mathbb{R}$ given by
  $$
  H(h)=
  \begin{cases}
       1-Nh & h\in [0,\tfrac{1}{s_{N+1}^2})  \\
     (\tfrac{1}{2} s^2_{N+1}-N)h+ \tfrac{1}{2s^2_{N+1}h} & h\in [\tfrac{1}{s_{N+1}^2}, \infty).
   \end{cases}
  $$
  It is readily seen than $H$ is a differentiable convex function, decreasing on its first piece, and the above minimizer $h^*$ can be found by solving $H^\prime(h^\star)=0$ on the second piece.
  
 Plugging this optimal $h^*$ in the rate Theorem~\ref{T2} completes the proof.
\end{proof}
A simpler, slightly weaker bound is obtained using $1 - \frac{2N}{s_{N+1}^2} = \frac{s_{N+1}^2-2N}{s_{N+1}^2} \le \frac{2+\frac12 \log(N)}{2N+2}$, leading to \[  f(x^{N+1})- f(x^\star) \leq \frac{BR}{\sqrt{N+1}} \sqrt{1 + \tfrac14 \log(N)} = O(\sqrt{\tfrac{\log N}{N}}) \]
which emphasizes the logarithmic loss compared to the best-iterate convergence rate.


\section{Subgradient method with constant step lengths}\label{Sec_f3}

Identifying a bound $B$ on the maximum norm of any subgradient may be difficult. Alternatively, one can adapt the subgradient method from Algorithm \ref{SM} such that constant step lengths are used. We express this constant step length as a fraction of the initial distance $t R$, for some value of $t > 0$. Since the length of a step is equal to $\| h_k g_k\|$, this implies the choice of step sizes $h_k = \frac{t R}{\| g_k\|}$ for each $k$. Algorithm \ref{SM_L} below presents the subgradient method with constant step length.

\begin{algorithm}
\caption{Projected subgradient method with constant step lengths}
\begin{algorithmic}
\smallskip
\State \textbf{Parameters:} number of iterations $N$, step length parameter $t>0$
\smallskip
\State \textbf{Inputs:} convex set $X$, convex function $f$ defined on $X$ with $B$-bounded subgradients, initial iterate $x^1 \in X$ satisfying $\|x^1 - x^*\| \le R$ for some minimizer $x^*$.
\smallskip
\State For $k=1, 2, \ldots, N$ perform the following steps:\\
\begin{enumerate}
\item
Select a subgradient $g^k\in \partial f(x^k)$.
\item
Compute $x^{k+1}=P_X\left(x^k- t \frac{R}{\|g^k\|} g^{k}\right)$.
\end{enumerate}
\smallskip
\State \textbf{Output:} last iterate $x_{N+1}$ 
\end{algorithmic}
\label{SM_L}
\end{algorithm}

We give below a convergence rate for Algorithm \ref{SM_L} by using Lemma \ref{P_SM}.

\begin{theorem}\label{T_c_s_l}
Let $f$ be a convex function with $B$-bounded sugradients on a convex set $X$. Consider $N$ iterations of Algorithm \ref{SM_L} with step length parameter $t>0$, starting from an initial iterate $x^1 \in X$ satisfying $\|x^1 - x^*\| \le R$ for some minimizer $x^*$. 
We have that the last iterate $x_{N+1}$ satisfies the following rate
  \begin{enumerate}[i)]
   \item
   If  $t\in (0, \tfrac{1}{s_{N+1}^2}]$, then
 \begin{align*}
f(x^{N+1})-f(x^\star)\leq BR\left( 1-Nt \right).
\end{align*}
   \item
 If  $t\in [\tfrac{1}{s_{N+1}^2}, \infty)$, then
\begin{align*}
f(x^{N+1})-f(x^\star)\leq BR\left( (\tfrac{1}{2} s^2_{N+1}-N)t+ \tfrac{1}{2s^2_{N+1}t}\right).
\end{align*}
 \end{enumerate}
\end{theorem}
\begin{proof}
We employ Lemma \ref{P_SM} to establish the theorem. Let $h_{N+1}=\tfrac{tR}{B}$ and $v_{N+1}={\alpha}$ for some $\alpha\geq 1$. Suppose that $h_k=\tfrac{tR}{\|g^k\|}$, $k\in\{1, \dots, N\}$. We define $v^k$ recursively as follows,
\begin{align}\label{fff}
v_k=\frac{h_{N+1}}{\sum_{i=k+1}^{N+1} h_iv_i},  \ \ \ \ \ k\in\{N, \dots, 1, 0\}.
\end{align}
It is seen that $0<v_0\leq v_1\leq \dots \leq v_N\leq v_{N+1}$. For $k\in\{1, \dots, N\}$, we have 
\begin{align*}
h_kv_{k}^2-(v_k-v_{k-1})\sum_{i=k}^{N+1}  h_iv_{i}=v_{k-1}\sum_{i=k}^{N+1}  h_iv_{i}-v_k\sum_{i=k+1}^{N+1}  h_iv_{i}=0.
\end{align*}
Furthermore,
 $$
 h_{N+1}v_{N+1}^2-(v_{N+1}-v_{N}) h_{N+1}v_{N+1} =h_{N+1}, \ \ v_0 \sum_{k=1}^{N+1}  h_kv_{k}=h_{N+1}.
 $$
 On the other hand, by \eqref{fff}, we get $v_{N}=\tfrac{1}{\alpha}$ and
 \begin{align*}
\tfrac{1}{v_k}=\tfrac{h_{k+1}}{h_{N+1}}v_{k+1}+\tfrac{1}{v_{k+1}}\geq v_{k+1}+\tfrac{1}{v_{k+1}}, \ \ \ k\in\{0, \dots, N-1\},
\end{align*}
where the last inequality follows from $\|g^{k+1}\|\leq B$. Since function $\mu: [1, \infty)\to \mathbb{R}$ given by $\mu(o)=\gamma o+\tfrac{1}{o}$ for $\gamma\geq 1$ is increasing on its domain, one can infer by induction that  
 \begin{align*}
v_k\leq \frac{1}{s_{\alpha,N+1-k}}, \ \ \ k\in\{0, \dots, N\}.
\end{align*}
Hence, by using Lemma \ref{P_SM}, we obtain
 \begin{align*}
f(x^{N+1})-f(x^\star)&\leq \tfrac{v_0^2}{2h_{N+1}}\left \|x^1-x^\star\right\|^2+\tfrac{t^2R^2}{2h_{N+1}}\sum_{k=1}^N v_k^2 +\tfrac{h_{N+1}B^2v_{N+1}^2}{2}\\
& \leq BR\left( \tfrac{1}{2}\left(s_{\alpha, N+1} \sqrt{t}-\tfrac{1}{s_{\alpha, N+1} \sqrt{t}}\right)^2+1-Nh\right),
\end{align*}
where the last inequality resulted from Proposition \ref{P_RS}. The rest of  the proof is analogous to Theorem \ref{T2}. 
\end{proof}


\section{Optimal subgradient methods}\label{Sec_opt}
\subsection{Lower bound on last-iterate convergence rate}
The convergence rate of any black-box method that relies on subgradients cannot be arbitrarily small. More precisely, for any method that moves at each iteration in a direction belonging to the span of the current and past subgradients, it is known that the accuracy of last iterate must obey a lower bound of the order $\Omega(\frac{1}{\sqrt{N}})$.
Nesterov \cite[Theorem 3.2.1]{Nesterov} proposes a Lipschitz continuous function $f$ with modulus $B>0$, for which any subgradient method that calls the first-order oracle $N$ times satisfies
$$
f(x^{N+1})-f(x^\star)\geq \frac{BR}{2(2+\sqrt{N+1})},
$$
where $\|x^1-x^\star\|\leq R$. Drori and Teboulle \cite{drori2016optimal} improved the above-mentioned lower bound and proposed the following lower bound
\begin{align}\label{Lower}
f(x^{N+1})-f(x^\star)\geq \frac{BR}{\sqrt{N+1}}.
\end{align}

A subgradient method for which the last-iterate accuracy would match this lower bound $\Omega(\frac{1}{\sqrt{N}})$ is certainly desirable \cite{shamir2012open}. Recently, Jain et al. \cite{jain2019making} introduced such a subgradient method when the feasible set $X$ is bounded. Indeed, they \cite[Theorem 2.6]{jain2019making} derive the following convergence rate for their proposed algorithm
$$
f(x^N+1)-f(x^\star)\leq \frac{15BD}{\sqrt{N+1}},
$$
where $D=\max_{x,y\in X} \|x-y\|$.

\subsection{Optimal subgradient method}
In this section, we introduce Algorithm \ref{OSM}, a subgradient method based on a new sequence of step sizes for which the last-iterate convergence rate matches the lower bound \eqref{Lower}.

\begin{algorithm}
\caption{Optimal projected subgradient method}
\begin{algorithmic}
\smallskip
\State \textbf{Parameters:} number of iterations $N$
\smallskip
\State \textbf{Inputs:} convex set $X$, convex function $f$ defined on $X$ with $B$-bounded subgradients, initial iterate $x^1 \in X$ satisfying $\|x^1 - x^*\| \le R$ for some minimizer $x^*$.
\smallskip
\State For $k=1, 2, \ldots, N$ perform the following steps:\\
\begin{enumerate}
\item
Select a subgradient $g^k\in \partial f(x^k)$.
\item
Compute $x^{k+1}=P_X\left(x^k- h_k  g^{k}\right)$ using step size step $h_k=\tfrac{R(N+1-k)}{B\sqrt{(N+1)^3}}$.
\end{enumerate}
\smallskip
\State \textbf{Output:} last iterate $x_{N+1}$ 
\end{algorithmic}
\label{OSM}
\end{algorithm}

In what follows, we establish that Algorithm \ref{OSM} attains the optimal rate of convergence. 
Indeed, the subsequent theorem presents a convergence rate for Algorithm \ref{OSM} by employing Lemma \ref{P_SM}.

\begin{theorem}\label{T_OSM}
Let $f$ be a convex function with $B$-bounded sugradients on a convex set $X$. Consider $N$ iterations of Algorithm \ref{OSM} starting from an initial iterate $x^1 \in X$ satisfying $\|x^1 - x^*\| \le R$ for some minimizer $x^*$. 
We have that the last iterate $x_{N+1}$ satisfies
\begin{align}\label{R_SMO}
f(x^{N+1})-f(x^\star)\leq \frac{BR}{\sqrt{N+1}}.
\end{align}
\end{theorem}
\begin{proof}
Suppose that $v_k$'s are given as follows,
$$
v_k=\frac{(N+1)^{\tfrac{3}{4}}}{N+1-k}\sqrt{\frac{B}{R}}, \ \ \ \ \ k\in \{0, \dots, N\},
$$
and $v_{N+1}=v_N$. It is seen that $0<v_0\leq v_1\leq \dots \leq v_{N}\leq v_{N+1}$. In addition, let $h_{N+1}=\tfrac{R}{B\sqrt{(N+1)^3}}$. For $k\in\{1, \dots, N\}$, we have 
$$
 h_kv_{k}^2 -(v_k-v_{k-1})\sum_{i=k}^{N+1}  h_iv_{i} =\tfrac{1}{N+1-k}-(N+2-k)(\tfrac{1}{N+1-k}-\tfrac{1}{N+2-k})=0.
$$
In addition,
$$
v_0 \sum_{k=1}^{N+1}  h_kv_{k}=\tfrac{1}{N+1} \sum_{k=1}^{N+1}  1=1, \ \ \ h_{N+1}v_{N+1}^2=1.
$$
By  Lemma \ref{P_SM}, we get
\begin{align*}
f(x^{N+1})-f(x^\star) & \leq  \tfrac{R}{2B\sqrt{(N+1)^3}}\sum_{i=1}^{N+1} \left\|g^i\right\|^2+
\tfrac{B}{2R\sqrt{N+1}}\left\|x^1-x^\star\right\|^2\\
& \leq  \tfrac{R}{2B\sqrt{(N+1)^3}}\sum_{i=1}^{N+1} B^2+
\tfrac{B}{2R\sqrt{N+1}}R^2=\frac{BR}{\sqrt{N+1}},
\end{align*}
 and the proof is complete.
\end{proof}

\subsection{Absence of optimal sugradient method with universal sequence of step sizes}
It is seen that the step sizes in Algorithm \ref{OSM} are dependent on the number of iterations, $N$. As conjectured in \cite{jain2019making}, it is not possible to introduce $\{h_k\}$ for which \eqref{R_SMO} holds for any arbitrary $N$. Before we show this point, we need to present a lemma.  

\begin{lemma}\label{L_last}
Consider Algorithm \ref{SM} with $h_1=\tfrac{1}{2\sqrt{2}}$ and $N=2$. 
\begin{enumerate}[i)]
  \item
 If $h_2\in (0, \tfrac{1}{8\sqrt{2}}]$, then there exists $f\in\mathcal{F}(\mathbb{R}^2)$ with $1$-bounded sugradients and $x^1$ and such that 
 $$
 f(x^3)-f^\star=\tfrac{1}{\sqrt{2}}-h_2,
 $$
 where $\|x^1-x^\star\|\leq 1$.
 
  \item
  If $h_2\in (\tfrac{1}{8\sqrt{2}}, \infty)$, then there exists $f\in\mathcal{F}(\mathbb{R}^3)$ with $1$-bounded sugradients and $x^1$ and such that 
 $$
 f(x^3)-f^\star=h_2+\tfrac{1}{64h_2}+\tfrac{16h_2}{(1+8\sqrt{2}h_2)^2},
 $$
 where $\|x^1-x^\star\|\leq 1$.

\end{enumerate}
\end{lemma}
\begin{proof}
First we establish $i)$. Let $f\in\mathcal{F}(\mathbb{R}^2)$ be given by 
$$
f(x)= \max\left\{x_1-1, x_2-1 , -1 \right\}.
$$
It is readily seen that $x^\star=0$ is an optimal solution to problem $\min \ f(x)$. Algorithm \ref{SM} with initial point $x^1=\tfrac{1}{\sqrt{2}}(1, 1)^T$ may generate the following points
$$
x^2=x^1-\tfrac{1}{2\sqrt{2}}e_1,  \ \ \ x^3=x^1-\tfrac{1}{2\sqrt{2}}e_1-h_2 e_2.
$$
In addition, $f(x^3)-f^\star=\tfrac{1}{\sqrt{2}}-h_2$ and we introduce an optimization problem with the desired properties. 

Now, we prove $ii)$. Let $\gamma=\tfrac{32h_2}{(1+8\sqrt{2}h_2)^2}$. One can show that $\gamma\in [0, 1]$. Suppose that
{\small{
\begin{align*}
\xi^1=
\begin{pmatrix}
\gamma\\ -\sqrt{1-\gamma^2}\\ 0
\end{pmatrix},
\xi^2=
\begin{pmatrix}
\gamma\\ \tfrac{\sqrt{1-\gamma^2}}{8\sqrt{2}h_2}\\ -\sqrt{ (1-\gamma^2)(1-\tfrac{1}{128h_2^2}) } 
\end{pmatrix},
\xi^3=
\begin{pmatrix}
\gamma\\ \tfrac{\sqrt{1-\gamma^2}}{8\sqrt{2}h_2}\\ \sqrt{ (1-\gamma^2)(1-\tfrac{1}{128h_2^2}) }
\end{pmatrix}.
\end{align*}
}}
It is readily seen that $\|\xi^k\|=1$, $i\in\{1, 2, 3\}$. Let $z^1=e_1$ and 
$$
z^2=e_1-\tfrac{1}{2\sqrt{2}}\xi^1, \ \ \ 
z^3=e_1-\tfrac{1}{2\sqrt{2}}\xi^1-h_2\xi^2.
$$
Consider the following linear functions
\begin{align*}
\alpha_1(x)=\gamma+&\langle \xi^1, x-z^1\rangle,\ \ \ \ 
\alpha_2(x)=\gamma+\tfrac{1-\gamma^2}{32h_2}-\tfrac{\gamma^2}{2\sqrt{2}}+\langle \xi^2, x-z^2\rangle,\\
&\alpha_3(x)=h_2+\tfrac{1}{64h_2}+\tfrac{16h_2}{(1+8\sqrt{2}h_2)^2}+\langle \xi^3, x-z^3\rangle.
\end{align*}
We define $f:\mathbb{R}^3\to\mathbb{R}$ by 
$$
f(x)=\max\left\{0, \alpha_1(x), \alpha_2(x), \alpha_3(x)\right\}.
$$ 
One can show that $x^\star=0$ is an optimal solution to problem $\min \ f(x)$. By doing some algebra, one can check that Algorithm \ref{SM} with initial point $x^1=e_1$ may generate the following points
$x^2=z^2$ and  $x^3=z^3$. It is seen that $f(x^3)-f(x^\star)=h_2+\tfrac{1}{64h_2}+\tfrac{16h_2}{(1+8\sqrt{2}h_2)^2}$, and the proof is complete. 
\end{proof}

Now, we present an argument why there does not exist a sequence $\{h_k\}$ that satisfies the convergence rate \eqref{R_SMO} for any arbitrary $N$. By contradiction, assume that there exists such a $\{h_k\}$. For the convenience suppose that $B=R=1$. Due to the exactness of rates given in Theorem \ref{T2}, we have $h_1=\tfrac{1}{2\sqrt{2}}$. By Lemma \ref{L_last}, one can infer that a convergence rate of Algorithm \ref{SM} with $h_1=\tfrac{1}{2\sqrt{2}}$ and $N=2$ cannot be lower than $o=0.5785$. Note that $o$ is  computed by solving $\min_{h\geq 0} H(h)$, where $H$ is given by
$$
H(h)=
\begin{cases}
    \tfrac{1}{\sqrt{2}}-h & h\in [0, \tfrac{1}{8\sqrt{2}}] \\
    h+\tfrac{1}{64h}+\tfrac{16h}{(1+8\sqrt{2}h)^2} & h\in ( \tfrac{1}{8\sqrt{2}}, \infty) 
\end{cases}.
$$
On the other hand, $o>0.5775>\tfrac{1}{\sqrt{3}}$. Hence, it is not possible to have $\{h_k\}$ for which \eqref{R_SMO} holds for any $N$ in the setting of Algorithm \ref{SM}.

As seen the optimal sizes depend on the number of iterations in the setting of Algorithm \ref{SM}. We conjecture that the incorporation of suitable momentum terms in the subgradient method may lead to a universal optimal algorithm whose convergence rate of $O\left(\tfrac{BR}{\sqrt{N+1}}\right)$ would hold for all iterations. 

\subsection{Optimal projected subgradient method using step lengths}
In the last part of this section, we present Algorithm \ref{OSM_s_l}, an optimal subgradient method based on step lengths. 

\begin{algorithm}
\caption{Optimal projected subgradient method (step lengths)}
\begin{algorithmic}
\smallskip
\State \textbf{Parameters:} number of iterations $N$
\smallskip
\State \textbf{Inputs:} convex set $X$, convex function $f$ defined on $X$ with $B$-bounded subgradients, initial iterate $x^1 \in X$ satisfying $\|x^1 - x^*\| \le R$ for some minimizer $x^*$.
\smallskip
\State For $k=1, 2, \ldots, N$ perform the following steps:\\
\begin{enumerate}
\item
Select a subgradient $g^k\in \partial f(x^k)$.
\item
Compute $x^{k+1}=P_X\left(x^k- t_k \frac{g^k}{\|g^k\|} \right)$ with $t_k=\tfrac{R(N+1-k)}{\sqrt{(N+1)^3}}$.
\end{enumerate}
\smallskip
\State \textbf{Output:} last iterate $x_{N+1}$ 
\end{algorithmic}
\label{OSM_s_l}
\end{algorithm}

In the forthcoming theorem, we provide a convergence rate for Algorithm  \ref{OSM_s_l}.

\begin{theorem}
Let $f$ be a convex function with $B$-bounded sugradients on a convex set $X$. Consider $N$ iterations of Algorithm \ref{OSM_s_l} starting from an initial iterate $x^1 \in X$ satisfying $\|x^1 - x^*\| \le R$ for some minimizer $x^*$. 
We have that the last iterate $x_{N+1}$ satisfies
 \begin{align}
f(x^{N+1})-f(x^\star)\leq \frac{BR}{\sqrt{N+1}}.
\end{align}
\end{theorem}
\begin{proof}
The proof is analogous to that of Theorem \ref{T_c_s_l}. Let 
$$
u_{N+1}=(N+1)^{\tfrac{3}{4}} \sqrt{\tfrac{B}{R}}, \ \ \ h_{N+1}=\tfrac{R}{B\sqrt{(N+1)^3}},
$$
and let $h_k=\tfrac{t_k}{\|g^k\|}$, $k\in\{1, \dots, N\}$. Let us define  $u^k$ recursively in the following manner,
\begin{align}
u_k=\frac{1}{\sum_{i=k+1}^{N+1} h_iu_i},  \ \ \ \ \ k\in\{N, \dots, 1, 0\}.
\end{align}
It is readily seen that $0<u_0\leq u_1\leq \dots \leq u_N\leq u_{N+1}$, and
\begin{align*}
h_ku_{k}^2-(u_k-u_{k-1})\sum_{i=k}^{N+1}  h_iu_{i}=u_{k-1}\sum_{i=k}^{N+1}  h_iu_{i}-u_k\sum_{i=k+1}^{N+1}  h_iu_{i}=0, \ \ \ k\in\{1, \dots, N\}.
\end{align*}
In addition,
 $$
 h_{N+1}u_{N+1}^2-(u_{N+1}-u_{N}) h_{N+1}u_{N+1} =1, \ \ u_0 \sum_{k=1}^{N+1}  h_ku_{k}=1.
 $$
 Consider $v_k$ given in the proof of Theorem \ref{T_OSM}. As 
 \begin{align*}
v_k=\frac{1}{\sum_{i=k+1}^{N+1} v_ih_i\tfrac{\|g^k\|}{B}},  \ \ \ \ \ k\in\{0, 1, \dots, N\},
\end{align*}
 one can infer by induction that $u_k\leq v_k$, $k\in\{0, 1, \dots, N+1\}$. Thus, by Lemma \ref{P_SM}, we get
 \begin{align*}
f(x^{N+1})-f(x^\star)&\leq \tfrac{u_0^2}{2}\left \|x^1-x^\star\right\|^2+\tfrac{1}{2}\sum_{k=1}^{N} h_k^2u_k^2\left\|g^k\right\|^2+\tfrac{1}{2} h_{N+1}^2u_{N+1}^2B^2 \\
& =\tfrac{u_0^2}{2}\left \|x^1-x^\star\right\|^2+\tfrac{1}{2}\sum_{k=1}^{N} \left( \tfrac{t_k}{B}\right)^2u_k^2B^2+\tfrac{1}{2} h_{N+1}^2u_{N+1}^2B^2\\
& \leq \tfrac{B}{2R\sqrt{N+1}}R^2+\tfrac{R}{2B\sqrt{(N+1)^3}}\sum_{i=1}^{N+1} B^2
=\frac{BR}{\sqrt{N+1}},
\end{align*}
where the last inequality follows from $u_k\leq v_k$, $k\in\{0, 1, \dots, N+1\}$, and the proof is complete. 
\end{proof}

\section*{Conclusion}
Before concluding, we briefly explain how most of the theorems in this paper were initially discovered. We used the performance estimation (PEP) methodology \cite{drori2014performance, taylor2017smooth}, which allowed us to compute numerically the exact last-iterate convergence rate of subgradient methods applied to convex functions with bounded sugradients. Assuming $B=R=1$ without loss of generality, the values of the worst case accuracy for several choices of $N$ and $h$ were matched with explicit analytical expressions, which required some guesswork including the introduction of the sequence $\{ s_k \}$. The next step was to use the numerical values of the dual multipliers to identify PEP-style proofs of those convergence rates, then to guess analytical expressions for those multipliers. Finally, we observed that large parts of the obtained proofs could be simplified by rewriting them as Jensen-type inequalities. After further simplifications, this ultimately leads to the proof technique that was exposed in Section~\ref{Sec_1}.

To summarize, we have provided in this paper new convergence rates for the subgradient method with constant step sizes and constant step lengths, and proved their tightness. Additionally, we have presented two optimal subgradient methods that attains the most favorable convergence rate achievable among subgradient algorithms. As avenues for future research, it would be valuable to investigate the convergence analysis of the (stochastic) proximal subgradient method with respect to the last iterate by employing some result analogous to Lemma \ref{P_SM}. Moreover, deriving tighter convergence rates for the stochastic subgradient method may be of interest. 

\bibliographystyle{siamplain}
\bibliography{references}
\end{document}